\documentclass[12pt,english]{amsart}
\usepackage[a4paper]{geometry}
\usepackage{amssymb,amsmath}
\usepackage{amsthm}
\usepackage{enumerate}

\numberwithin{equation}{section}

\newtheorem*{Thmuncount}{Theorem}
\newtheorem{theorem}{Theorem}[section]

\newtheorem{proposition}[theorem]{Proposition}
\newtheorem{lemma}[theorem]{Lemma}
\newtheorem{corollary}[theorem]{Corollary}
\newtheorem{Definition}[theorem]{Definition}
\newenvironment{definition}{\begin{Definition}\rm}{\end{Definition}}
\newtheorem{Remark}[theorem]{Remark}
\newenvironment{remark}{\begin{Remark}\rm}{\end{Remark}}
\newtheorem{RHproblem}[theorem]{RH problem}

\newtheorem{Example}[theorem]{Example}

\theoremstyle{definition}
\newtheorem{question}{Question}
\newtheorem{fait}{Claim}

\theoremstyle{plain}

\newtheorem*{maintheo}{Theorem A}
\newtheorem*{theoA}{Theorem B}
\newtheorem*{theoC}{Theorem C}

\newcommand{\C}{\mathbb{C}}

\newcommand{\Z}{\mathbb{Z}}
\newcommand{\N}{\mathbb{N}}
\newcommand{\R}{\mathbb{R}}
\newcommand{\Q}{\mathbb{Q}}
\newcommand{\T}{\mathbb{T}}

\newcommand{\orb}{\text{Orb}}
\newcommand{\aorb}{\overline{\text{Orb}}}
\newcommand{\eorb}{\text{\emph{Orb}}}
\newcommand{\aeorb}{\overline{\text{\emph{Orb}}}}

\renewcommand{\bar}{\overline}
\renewcommand{\tilde}{\widetilde}

\usepackage[bookmarksopen, naturalnames]{hyperref}

\usepackage[english]{babel}

\makeatother

\begin{document}

\title[{$\Gamma$}-supercyclicity]{{$\Gamma$}-supercyclicity}
\author{S. Charpentier, R. Ernst, Q. Menet}
\address{St\'ephane Charpentier, 
Institut de Math\'ematiques, UMR 7373, Aix-Marseille Universit\'e, 39 rue F. Joliot Curie, 13453 Marseille Cedex 13, FRANCE}
\email{stephane.charpentier.1@univ-amu.fr}
\address{Romuald Ernst, 
Institut de Math\'ematiques, UMR 7373, Aix-Marseille Universit\'e, 39 rue F. Joliot Curie, 13453 Marseille Cedex 13, FRANCE}
\email{ernst.r@math.cnrs.fr}
\address{Quentin Menet, 
Laboratoire de Math\'ematiques de Lens, Universit\'e d'Artois, Rue Jean Souvraz S.P. 18 , 62307 Lens, FRANCE}
\email{quentin.menet@univ-artois.fr}
\begin{abstract}
We characterize the subsets $\Gamma$ of $\C$ for which the notion of $\Gamma$-supercyclicity coincides with the notion of hypercyclicity, where an operator $T$ on a Banach space $X$ is said to be $\Gamma$-supercyclic if there exists $x\in X$ such that $\overline{\text{Orb}}(\Gamma x, T)=X$. In addition we characterize the sets $\Gamma \subset \C$ for which, for every operator $T$ on $X$, $T$ is hypercyclic if and only if there exists a vector $x\in X$ such that the set $\text{Orb}(\Gamma x, T)$ is somewhere dense in $X$. This extends results by Le\'on-M\"uller and Bourdon-Feldman respectively. We are also interested in the description of those sets $\Gamma \subset \C$ for which $\Gamma$-supercyclicity is equivalent to supercyclicity.\end{abstract}

\maketitle

\section{Introduction and statements of the main results}

\subsection{Introduction}

Let $X$ be a complex Banach space and let $L(X)$ denote the space of bounded linear operators on $X$. For $T$ in $L(X)$, $x$ in $X$, and $\Gamma$ a non-empty subset of the complex plane $\C$, we denote $\text{Orb}(\Gamma x,T)=\left\{ \gamma T^n x:\,\gamma \in \Gamma,\, n\geq 0\right\}$. We say that $x$ is \emph{$\Gamma$-supercyclic} for $T$ if $\text{Orb}(\Gamma x,T)$ is dense in $X$ and $T$ will be said to be $\Gamma$-supercyclic if it admits a $\Gamma$-supercyclic vector. In particular, if $\Gamma=\mathbb{C}$, $x$ $\Gamma$-supercyclic for $T$ reads $x$ supercyclic for $T$ and if $\Gamma$ reduces to a single nonzero point, $x$ $\Gamma$-supercyclic for $T$ reads $x$ hypercyclic for $T$.
The notion of hypercyclicity was already studied by Birkhoff in the twenties but it really began to attract much attention in the late seventies. The terminology follows that of supercyclicity, introduced by Hilden and Wallen \cite{Hilwal} in the early seventies, and the former notion of cyclicity. While the latter is directly connected with the well-known Invariant Subspace Problem, hypercyclicity is connected with the Invariant \emph{Subset} Problem. To learn much about linear dynamics, we refer to the very nice books \cite{Bay,Grope}.

One of the first important results was Kitai Criterion \cite{Kitai}, refined by B\`es \cite{Bes} in the following form and known as the Hypercyclicity Criterion.

\begin{Thmuncount}[\textbf{Hypercyclicity Criterion}]Let $T\in L(X)$. We assume that there exist two dense subsets $X_0,Y_0\subset X$, an increasing sequence $(n_k)_k\subset \mathbb{N}$, and maps $S_{n_k}:Y_0\rightarrow X$ such that for any $x\in X_0$ and $y\in Y_0$ the following holds:
\begin{enumerate}\item $T^{n_k}x\rightarrow 0$ as $k\rightarrow \infty$;
\item $S_{n_k}y\rightarrow 0$ as $k\rightarrow \infty$;
\item $T^{n_k}S_{n_k}y\rightarrow y$ as $k\rightarrow \infty$.
\end{enumerate}
Then $T$ is hypercyclic.
\end{Thmuncount}

We mention that there also exists a so-called Supercyclicity Criterion due to Salas \cite{Salas}, which is easily seen to be non-necessary for supercyclicity. The Hypercyclicity Criterion gives an effective way of proving that an operator is hypercyclic and covers a wide range of \emph{concrete} hypercyclic operators, allowing to directly recover historical examples of hypercyclic operators exhibited by Birkhoff \cite{Birk}, MacLane \cite{MacLane} or Rolewicz \cite{Rolewicz}. A long-standing major open question was to know whether the Hypercyclicity Criterion is necessary for an operator to be hypercyclic. B\`es and Peris \cite{Besper} observed that satisfying the Hypercyclicity Criterion is in fact equivalent to being hereditarily hypercyclic or weakly mixing and, in 2008, De La Rosa and Read built a Banach space and a non-weakly mixing hypercyclic operator acting on this space, giving a negative answer to the above mentioned question. A bit later, Bayart and Matheron \cite{Baymathyppb} provided such an 
example on many classical Banach spaces, including the separable Hilbert space. We also refer to \cite{Berbonpe} where the authors show equivalence between the Hypercyclicity Criterion and other criteria, and unify different versions of the Supercyclicity Criterion.

Anyway the existence of effective characterizations of hypercyclicity and supercyclicity has been a wide subject of interest and several results have been given. In this direction, Herrero \cite{Herrero} conjectured that an operator $T$ needs to be hypercyclic if we only assume that the orbit under $T$ of some finite set of vectors is dense in $X$. In 2000, Peris \cite{Permultihyp} and Costakis \cite{Costakisconjherrero} independently gave a positive answer to Herrero's conjecture. In 2004 Le\'on and M\"uller proved another result in the same spirit \cite[Corollary 2]{LeonMuller}.
\begin{Thmuncount}[\textbf{Le\'on-M\"uller Theorem}]\label{LM}
Let $T\in L(X)$. Then $x\in X$ is hypercyclic for $T$ if and only if $x$ is $\T$-supercyclic for $T$.
\end{Thmuncount}
Here, it is remarkable that we can replace the orbit of a single vector by the orbit of an uncountable set of vectors. Nevertheless it is worth noting that this uncountable set of vectors is one dimensional and that a specific group structure is underlying. We mention that, as a corollary of this theorem, Le\'on and M\"uller proved that for any complex number $\lambda$ with modulus $1$, $T$ is hypercyclic if and only if $\lambda T$ is hypercyclic (with same hypercyclic vectors), answering another question posed by Herrero \cite{Herrero}.

Roughly speaking the two previous results refer to the general problem of how \emph{big} can be a set with dense orbit in order to still ensure hypercyclicity. So, in the context of Le\'on-M\"uller Theorem the following question arises.

\begin{question}
Is it possible to characterize the sets $\Gamma\subset \C$ such that $T$ is hypercyclic if and only if $T$ is $\Gamma$-supercyclic? 
\end{question}

The proof of Le\'on-M\"uller Theorem heavily relies on the group (or rather semigroup) structure of $\T$ with respect to the complex multiplication. This group-theoretic approach has been deeply developed by Shkarin \cite{Shkauniv} and Matheron \cite{Matheronsub} (see also \cite[Chapter 3]{Bay}) in a much more general and abstract setting, tending to suggest that it is inevitable. 

\medskip{}

A different approach for characterizing hypercyclicity in an (apparently) weaker way may consist in considering how \emph{small} can be the orbit of a given vector under $T$ to still ensure that $T$ is hypercyclic. The first result in this direction is obtained by Feldman \cite{Feldmanperturbations} in 2002 who proved that an operator is hypercyclic if and only if there exists $d>0$ and a vector $x\in X$ having a $d$-dense orbit, where a set is said to be $d$-dense if it intersects any open ball of radius $d$. Moreover, Feldman also proved that a vector with $d$-dense orbit is not necessarily a hypercyclic vector. In the same year Bourdon and Feldman \cite{Boufel} proved a very nice result:
\begin{Thmuncount}[\textbf{Bourdon-Feldman Theorem}]
Let $X$ be a Banach space, $T\in L(X)$ and $x\in X$. Then $x$ is hypercyclic for $T$ if and only if $\eorb(x,T)$ is somewhere dense in $X$.
\end{Thmuncount}
We remark that Peris and Costakis' result is a corollary of the latter. Later on, Bayart and Matheron \cite{Bay} used the group-theoretic approach initiated by Le\'on and M\"uller to extend Bourdon-Feldman Theorem to a general framework involving semigroups. We just quote their result in a peculiar case, that is $\mathcal{T}=\{\lambda T^n;\,\lambda \in \T,n\in\N\}$ using their notations.

\begin{Thmuncount}[Theorem 3.13 of \cite{Bay}]
Let $X$ be a complex Banach space, $T\in L(X)$ and $x\in X$. Then $x$ is hypercyclic for $T$ if and only if $\eorb(\T x,T)$ is somewhere dense in $X$.
\end{Thmuncount}

Here again, the following question naturally arises.

\begin{question}
Is it possible to characterize the sets $\Gamma\subset \C$ such that $x$ is hypercyclic for $T$ if and only if $\orb(\Gamma x,T)$ is somewhere dense?
\end{question}

\medskip{}

Similar questions have been addressed in the context of supercyclicity where it makes sense to consider the finite dimensional setting. A reasonable question is to find a \emph{small} set $\Gamma\subset \C$ such that $T$ is supercyclic if and only if $T$ is $\Gamma$-supercyclic. However supercyclicity allows more exoticism, for example in terms of spectral properties of supercyclic operators. We recall that if $T$ is supercyclic then the point spectrum $\sigma_p(T^*)$ of the adjoint of $T$ contains at most one nonzero point and that for any nonzero complex number $\alpha$ there exists a supercyclic operator $T$ with $\sigma_p(T^*)=\alpha$. In 2001, Montes-Rodr\'iguez and Salas \cite{Montsalas} proved that if $T$ satisfies the Supercyclicity Criterion then $\sigma_p(T^*)$ is empty and $T$ needs to be $\R^+$-supercyclic (sometimes called positive supercyclic \cite{Leon}). Later, in \cite{LeonMuller}, the same crucial tool used to prove Le\'on-M\"uller Theorem allows the authors to show the equivalence between supercyclicity and $\R^+$-supercyclicity  whenever $\sigma _p(T^*)=\emptyset$. This result complemented the previous work of Berm\'udez, Bonilla and Peris \cite{Berbonper} who proved that $T$ is $\R$-supercyclic if and only if $T$ is $\R^+$-supercyclic, whatever the spectrum of $T^*$, and also provided counterexamples to the equivalence with supercyclicity when $\sigma_p(T^*)\neq\emptyset$. Finally, answering a question by Le\'on and M\"uller, Shkarin used his abstract group-theoretic approach to characterize the operators for which the equivalence between $\R^+$-supercyclicity and supercyclicity holds \cite{Shkauniv}.
\begin{Thmuncount}[\textbf{Shkarin Theorem}]
Let $X$ be a complex Banach space and $T\in L(X)$. $T$ is $\R^+$-supercyclic if and only if $T$ is supercyclic and either the point spectrum $\sigma_p(T^*)$ is empty or $\sigma_p(T^*)=\{re^{i\theta}\}$, with $r\neq 0$ and  $\theta \in\pi(\R\setminus \Q)$.
\end{Thmuncount}

More generally the following question arises, an answer to which would involve the spectrum of $T^*$.

\begin{question}
For which $\Gamma\subset\C$ is supercyclicity equivalent to $\Gamma$-supercyclicity?
\end{question} 

\subsection{Statements of the main results}

The purpose of this article is to discuss Questions 1, 2 and 3. In order to deal with Questions 1 and 2 we introduce two properties that a subset of $\C$ can enjoy or not.
\begin{definition}\label{prop-P}Let $\Gamma$ be a subset of $\C$.
\begin{enumerate}\item $\Gamma$ is said to be a \textbf{hypercyclic scalar set} if the following holds true: For every infinite-dimensional complex Banach space $X$, every $T\in L(X)$ and every $x\in X$
$$\overline{\text{Orb}}(\Gamma x,T)=X\text{ if and only if }x\text{ is hypercyclic for }T.$$
\item $\Gamma$ is said to be a \textbf{somewhere hypercyclic scalar set} if the following holds true: For every infinite-dimensional complex Banach space $X$, every $T\in L(X)$ and every $x\in X$
$$\text{Orb}(\Gamma x,T)\text{ is somewhere dense in }X\text{ if and only if }x\text{ is hypercyclic for }T.$$
\end{enumerate}
\end{definition}
Obviously, if $\Gamma$ is a somewhere hypercyclic scalar set then $\Gamma$ is a hypercyclic scalar set and if $\Gamma$ is a hypercyclic (resp. somewhere hypercyclic) scalar set then any smaller set is also a hypercyclic (resp. somewhere hypercyclic) scalar set. According to Le\'on-M\"uller Theorem and the refinement of Bourdon-Feldman Theorem stated above, $\T$ is a somewhere hypercyclic scalar set.

We provide complete answers to Question 1 (Theorem~A) and Question 2 (Theorem~B):
\begin{maintheo}\label{mainthm}A non-empty subset $\Gamma$ of $\C$ is a hypercyclic scalar set if and only if $\Gamma\setminus\{0\}$ is non-empty, bounded and bounded away from $0$.
\end{maintheo}
In Theorem~B below, we denote by $\Gamma \T$ the set $\{\gamma z:\,\gamma \in \Gamma,\,|z|=1\}$.
\begin{theoA}\label{thm-Pprime}A non-empty subset $\Gamma$ of $\C$ is a somewhere hypercyclic scalar set if and only if $\overline{\Gamma \T}\setminus\{0\}$ is non-empty, bounded and bounded away from $0$ and has an empty interior.
\end{theoA}

In view of Theorems A and B, $\Gamma$ is a somewhere hypercyclic scalar set if and only if $\Gamma$ is a hypercyclic scalar set and $\overline{\Gamma\T}$ has an empty interior.

\begin{remark}\label{rem-MT}(1) Theorems A and B can be stated in a slightly different way in the real setting, namely for real Banach spaces. Since it makes no particular difficulties to adapt the above statements and the corresponding proofs given in the remaining of the paper, we leave the details to the reader.\\
(2) Using Theorem~A and counterexamples given in Section 2, we can observe that $\Gamma$ is a hypercyclic scalar set if and only if it satisfies the following: For every infinite-dimensional complex Banach space $X$, every $T\in L(X)$,
$$T\text{ is hypercyclic if and only if }T\text{ is }\Gamma\text{-supercyclic}.$$
(3) Similarly $\Gamma$ is a somewhere hypercyclic scalar set if and only if it satisfies the following: For every infinite-dimensional complex Banach space $X$, every $T\in L(X)$,
$$T\text{ is hypercyclic if and only if }\text{Orb}(\Gamma x,T)\text{ is somewhere dense in }X\text{ for some }x\in X.$$
\end{remark}
For example any ring $\{\mu\lambda,\,\mu \in [a,b],\,\lambda \in \T\}$ with $0<a< b<+\infty$ is a hypercyclic scalar set but not a somewhere hypercyclic scalar set. The same holds for any closed paths in $\C$ which do not contain $0$, different from a nonzero multiple of the unit circle $\T$. Moreover a union of circles centered at $0$ with bounded and bounded away from zero radii is a somewhere hypercyclic scalar set if and only if the family of these radii has an empty interior.

\medskip{}

Furthermore we give a general anwser to Question 3 when $\sigma_p(T^*)$ is non-empty.

\begin{theoC}Let $X$ be a complex Banach space.
\begin{enumerate}
\item For every $\theta\in\R$, $\Gamma \subset \C$ satisfies the property: "For every $T\in L(X)$ with $\sigma_p(T^*)=\{e^{i\theta}\}$ and every $x\in X$
$$x\text{ is }\Gamma\text{-supercyclic for }T\text{ if and only if }x\text{ is supercyclic for }T\text{"}$$
if and only if $\Gamma G_{\theta}$ is dense in $\C$, where $G_{\theta}$ stands for the subgroup of $\T$ generated by $e^{i\theta}$.
\item For any $r\neq 1$ and any $\theta\in\R$, there exist $T\in L(X)$ with $\sigma_p(T^*)=\{re^{i\theta}\}$ and $\Gamma \subset \C$ satisfying $\overline{\Gamma G_{\theta}}=\C$, such that $T$ is supercyclic but not $\Gamma$-supercyclic.
\end{enumerate}
\end{theoC}

\begin{remark}(1)
For every $\theta \in\R$, Theorem~C provides a complete understanding of the problem of describing those subsets $\Gamma$ of $\C$ such that for every $X$, every $T\in L(X)$ with $\sigma_p(T^*)=\{e^{i\theta}\}$ and every $x\in X$,
$$x\text{ is }\Gamma\text{-supercyclic for }T\text{ if and only if }x\text{ is supercyclic for }T.$$
Nevertheless such a problem for every $T$ (namely for those $T$ such that $\sigma_p(T^*)=\emptyset$ or $\sigma_p(T^*)=\{re^{i\theta}\}$ with $r\neq 1$) remains unclear.\\
(2) One shall notice that, when $\theta \in \pi(\R\setminus\Q)$, the condition $\overline{\Gamma G_{\theta}}=\C$ is equivalent to $\overline{\Gamma \T}=\C$.\\
(3) It is worth observing that the equivalence between $\Gamma$-supercyclicity and supercyclicity for $T$ with $\sigma_p(T^*)=re^{i\theta}$, $\theta \in \pi(\R\setminus \Q)$, depends on $r$. This differs from the particular case of the $\R_+$-supercyclicity which is treated in Shkarin Theorem.\\
(4) It may also be interesting to notice that, in (1), the equivalence between $x$ $\Gamma$-supercyclic and $x$ supercyclic not only depends on the fact that $\theta$ is a rational or an irrational multiple of $\pi$ (as in Shkarin Theorem) but also depends on every $\theta\in\pi\Q$. Indeed if $\theta,\theta'\in\pi\Q\cap]0;\pi[$ with $\theta<\theta'$ and $\Gamma=\{re^{i\alpha}, r\in\R^{+}, \alpha\in[0;\theta]\}$, then $\Gamma G_{\theta}=\C$ while $\overline{\Gamma G_{\theta'}}\neq\C$.\\
(5) Observe that contrary to Theorems A and B, it makes sense to also consider finite dimensional Banach spaces $X$ in Theorem~C.
\end{remark}

The article is organized as follows. The purpose of Section 2 is to show the necessity part of Theorem~A. In Section 3 we prove the sufficiency part of Theorem~A which is the most difficult one. Section 4 is devoted to the proof of Theorem~B and Section 5 to that of Theorem~C. 

\section{Theorem~A - Necessity part}

To prove that if $\Gamma$ is a hypercyclic scalar set then $\Gamma\setminus\{0\}$ is bounded and bounded away from zero, it suffices to exhibit examples of $\Gamma$-supercyclic operators that are not hypercyclic when $\Gamma\setminus\{0\}$ is not bounded or not bounded away from zero. Note that if $\Gamma$ is a hypercyclic scalar set then obviously $\Gamma\setminus\{0\}$ needs to be non-empty.

We begin by proving that if $\Gamma$ is a hypercyclic scalar set then $\Gamma$ is bounded.

\begin{proposition}\label{prop-nec-1}
Let $\Gamma$ be an unbounded subset of $\C$.
Then, the backward shift operator $B$ on $\ell^{2}(\N)$ is $\Gamma$-supercyclic but not hypercyclic.
\end{proposition}

\begin{proof}
It is well-known that $B$ is not hypercyclic, since the orbit of any vector is bounded. Let $\Gamma$ be an unbounded subset of $\C$. We are going to construct a $\Gamma$-supercyclic vector for $B$.

Let $\{y_k: k\in\N\}$ be a dense subset of $c_{00}$. We denote by $F$ the forward shift on $\ell^{2}(\N)$ and we let $d(y_k)=\max\{j\ge 0: y_k(j)\ne 0\}$. First, we construct by induction a sequence $(\gamma_{k})_{k\in\N}\subset\Gamma\setminus\{0\}$, and a sequence of integers $(m_k)_{k\in\N}$ such that for every $k\in\N$:
\begin{enumerate}[(i)]
 \item $\|\frac{1}{\gamma_k}y_k\Vert<2^{-k}$\label{nbcritI};
 \item For every $i<k$, $\frac{|\gamma_i|}{|\gamma_k|}\Vert y_k\Vert<2^{-k}$\label{nbcritII}.
 \item For every $i<k$, $m_k>m_i+d(y_i)$\label{nbcritIII};
\end{enumerate}
If $m_0,\dots,m_{k-1}$ and $\gamma_0,\cdots,\gamma_{k-1}$ have been chosen, we remark that it suffices to choose $\gamma_k$ sufficiently big in order to satisfy (i) and (ii) and to choose $m_k$ sufficiently big in order to satisfy (iii).

Thanks to (\ref{nbcritI}), we can let $x=\sum_{i=0}^{+\infty}\frac{1}{\gamma_i}F^{m_i}y_i$ since $\|F^{m_i}y_i\|=\|y_i\|$ and we claim that $x$ is $\Gamma$-supercyclic for $T$.
\\Let $k\in\N$. We have 
\begin{align*}
 \Vert\gamma_{k}B^{m_{k}}x-y_k\Vert&\leq\sum_{j<k}\Vert\frac{\gamma_{k}}{\gamma_j}B^{m_{k}}(F^{m_j}y_j)\Vert+\Vert \frac{\gamma_{k}}{\gamma_k}B^{m_{k}}(F^{m_k}y_k)-y_k\Vert+\sum_{j>k}\Vert\frac{\gamma_{k}}{\gamma_j}B^{m_{k}}(F^{m_j}y_j)\Vert&\\
&\leq \sum_{j<k}0+0+\sum_{j>k}\frac{|\gamma_{k}|}{|\gamma_j|} \|y_j\|\text{ by (\ref{nbcritIII})}&\\
&\leq \sum_{j>k} 2^{-j}=2^{-k}\underset{k\to+\infty}{\longrightarrow}0\text{ by (\ref{nbcritII}).}&\\
\end{align*}
\end{proof}

Since the backward shift has norm 1, we cannot hope using it to prove that $\Gamma\setminus \{0\}$ has to be bounded away from zero. Thus, we will use a bilateral shift instead.

\begin{proposition}\label{prop-nec-2}
Let $\Gamma$ be a subset of $\C$ and assume that $\Gamma \setminus \{0\}$ is not bounded away from zero.
Then, the shift operator $B_w$ on $\ell^{2}(\Z)$, with weight sequence $w_i=2$ if $i>0$ and $w_i=1$ else, is $\Gamma$-supercyclic but not hypercyclic.
\end{proposition}

\begin{proof}
It is well-known that $B_w$ is not hypercyclic, since the orbit of any non-zero vector is bounded away from zero.
Let $\Gamma$ be a subset of $\C$ such that $\Gamma\setminus\{0\}$ is not bounded away from zero. We are going to construct a $\Gamma$-supercyclic vector for $B_w$.

Let $\{y_k: k\in\N\}$ be a dense subset of $c_{00}(\mathbb{Z})$. We let $d(y_k)=\max\{j\ge 0: y_k(j)\ne 0\}$ and we denote by $F_{\frac{1}{w}}$ the inverse of $B_w$ on $\ell^{2}(\Z)$. In other words, $F_{\frac{1}{w}}$ is the forward weighted shift $F_{\nu}$ where $\nu_i=\frac{1}{2}$ if $i\ge 0$ and $\nu_i=1$ else.
First, we construct by induction a sequence $(\gamma_{k})_{k\in\N}\subset\Gamma\setminus\{0\}$ and a sequence of integers $(m_k)_{k\in\N}$ such that for every $k\in\N$:
\begin{enumerate}[(i)]
 \item $\Vert \frac{1}{\gamma_k} F^{m_k}_{\frac{1}{w}}y_k\Vert<2^{-k}$\label{nbcritIw};
 \item For every $i<k$, $\frac{|\gamma_i|}{|\gamma_k|}\Vert B_w^{m_i}F^{m_k}_{\frac{1}{w}}y_k\Vert<2^{-k}$\label{nbcritIIw}.
 \item For every $i<k$, $\frac{|\gamma_k|}{|\gamma_i|}\Vert B_w^{m_k}F^{m_i}_{\frac{1}{w}}y_i\Vert<2^{-k}$\label{nbcritIIIw};
\end{enumerate}
If $m_0,\dots,m_{k-1}$ and $\gamma_0,\cdots,\gamma_{k-1}$ have been chosen, we first remark that we can choose $\gamma_k$ sufficiently small such that for every $i<k$, we have $\frac{|\gamma_k|}{|\gamma_i|}\Vert B_w^{m}F^{m_i}_{\frac{1}{w}}y_i\Vert<2^{-k}$ for every $m\ge 0$ since 
\[\frac{|\gamma_k|}{|\gamma_i|}\Vert B_w^{m}F^{m_i}_{\frac{1}{w}}y_i\Vert\le \frac{|\gamma_k|}{|\gamma_i|}2^{m_i+d(y_i)}\|y_i\|.\]
We can then choose $m_k$ sufficiently big in order to satisfy (i) and (ii) since for every $y\in c_{00}(\mathbb{Z})$, we have $F^{m}y\to 0$ as $m\to \infty$.

Thanks to (\ref{nbcritIw}), we can let $x=\sum_{i=0}^{+\infty}\frac{1}{\gamma_i} F_{\frac{1}{w}}^{m_i}y_i$ and we claim that $x$ is $\Gamma$-supercyclic for $T$.
\\Let $k\in\N$. We have 
\begin{align*}
 \Vert\gamma_{k}B_w^{m_{k}}x-y_k\Vert&\leq\sum_{j<k}\Vert\frac{\gamma_{k}}{\gamma_j}B_w^{m_{k}}F_ {\frac{1}{w}}^{m_j}y_j\Vert+\Vert \frac{\gamma_{k}}{\gamma_k}B_w^{m_{k}}F_{\frac{1}{w}}^{m_k}y_k-y_k\Vert+\sum_{j>k}\Vert\frac{\gamma_{k}}{\gamma_j}B_w^{m_{k}}F_{\frac{1}{w}}^{m_j}y_j\Vert&\\
&\leq \sum_{j<k}\frac{1}{2^k}+0+\sum_{j>k}\frac{1}{2^j}\text{by (\ref{nbcritIIw}) and (\ref{nbcritIIIw})}&\\
&\leq \frac{k+1}{2^k}\underset{k\to+\infty}{\longrightarrow}0.&
\end{align*}
\end{proof}

Propositions \ref{prop-nec-1} and \ref{prop-nec-2} give the necessity part of Theorem~A.

\section{Theorem~A - Sufficiency part}

Let $X$ be an infinite-dimensional complex Banach space. In this section, we intend to prove the following theorem.

\begin{theorem}\label{main-thm-sufficiency}Let $T\in L(X)$ and $\Gamma \subset \C$ be such that $\Gamma\setminus\{0\}$ is non-empty, bounded and bounded away from~$0$. If $x$ is $\Gamma$-supercyclic for $T$ then $x$ is hypercyclic for $T$.
\end{theorem}

Let $T\in L(X)$ and let $x\in X$ be $\Gamma$-supercyclic for $T$. As it is clear that the point zero plays no role, in what follows we are going to suppose that $0\notin\Gamma$. Then, we notice that $\Gamma$ is included in some ring of the form $[a,b]\T$, with $0<a\leq b<+\infty$, and that $x$ is $[a,b]\T$-supercyclic whenever $x$ is $\Gamma$-supercyclic. In addition, up to a dilation, we see that $x$ is $[a,b]\T$-supercyclic if and only if $x$ is $[1,b/a]\T$-supercyclic. Therefore, to prove Theorem \ref{main-thm-sufficiency} we are reduced to prove the following

\begin{theorem}\label{main-thm-sufficiency-coronna}Let $T\in L(X)$ and $1\leq b < +\infty$. If $x$ is $[1,b]\T$-supercyclic for $T$ then $x$ is hypercyclic for $T$.
\end{theorem}

The proof of Theorem \ref{main-thm-sufficiency-coronna} relies on several lemmas.

\begin{lemma}\label{lemma0} Let $\Gamma\subset \mathbb{C}$ be a non-empty set bounded and bounded away from $0$, let $T\in L(X)$ and let $x$ be a $\Gamma$-supercyclic vector for $T$. Then for every $y\in X$, there exists $\gamma\in \overline{\Gamma}$ and an increasing sequence $(n_k)$ of integers such that $\gamma T^{n_k}x\to y$.
\end{lemma}

\begin{proof}
We first remark that $\orb(\Gamma x,T)$ has an empty interior by the Baire Category Theorem. Let $y\in X$. We deduce from the above assertion that there exists a sequence $(y_k)\subset \aorb(\Gamma x,T)\backslash\orb(\Gamma x,T)$ converging to $y$. We remark that if $z\in \aorb(\Gamma x,T)\backslash\orb(\Gamma x,T)$ then there exist an increasing sequence $(n_k)\subset \N$ and $(\gamma_j)\subset \Gamma$ such that $\gamma_j T^{n_j}x\to z$. We can thus construct a sequence $(\gamma_k)\subset \Gamma$ and an increasing sequence $(n_k)$ such that $\|y_k-\gamma_kT^{n_k}x\|<2^{-k}$. By using the compactness of $\overline{\Gamma}$ and the fact that $\overline{\Gamma}$ is bounded away from $0$, we then obtain the desired result.
\end{proof}

We now deduce the following corollary which is an immediate consequence of the previous lemma.

\begin{corollary}\label{cor0} Let $\Gamma\subset \mathbb{C}$ be a non-empty set bounded and bounded away from $0$, let $T\in L(X)$ and let $x$ be a $\Gamma$-supercyclic vector for $T$. Then for every $n\geq 0$, we have
$$\bigcup_{\gamma\in \bar{\Gamma}}\aeorb(\gamma T^nx,T)=X.$$
\end{corollary}

The following lemma is well-known in the case of hypercyclic operators, we provide here a slight generalization with its proof for the sake of completeness.

\begin{lemma}\label{lemma-spectre-vide}If $T$ is $\Gamma$-supercyclic with $\Gamma \subset \C$ non-empty, bounded and bounded away from $0$, then $p(T)$ has dense range for any nonzero polynomial $p$.
\end{lemma}

\begin{proof}It is a well-known fact that it suffices to prove that the point spectrum $\sigma _p(T^*)$ of the adjoint $T^*$ of $T$ is empty. Let $x$ be a $\Gamma$-supercyclic vector for $T$. By contradiction, assume that $\alpha \in \sigma _p(T^*)$ and let $y^* \in X^*\setminus\{0\}$ be such that $T^*y^*=\alpha y^*$. Since $x$ is $\Gamma$-supercyclic for $T$, we get
$$\C  =  \overline{\left\{\langle\gamma T^nx,y^*\rangle:\,\gamma \in \Gamma,\,n\geq 0\right\}} =  \overline{\left\{\gamma \alpha ^n:\,\gamma \in \Gamma,\,n\geq 0\right\}}\langle x,y^*\rangle.$$
If $|\alpha| \leq 1$ or $\langle x,y^*\rangle =0$ then the last set is bounded and cannot be dense in $\C$; if $|\alpha|>1$ and $\langle x,y^*\rangle \neq 0$ then it is bounded away from $0$, hence a contradiction.
\end{proof}

From this lemma, we can deduce the following one which shows that as soon as $T$ is $\Gamma$-supercyclic then the set of $\Gamma$-supercyclic vectors contains a particular dense linear subspace apart from zero. 

\begin{lemma}\label{lemma-P(T)}If $x$ is $\Gamma$-supercyclic for $T$ with $\Gamma \subset \C$ non-empty, bounded and bounded away from $0$, then $p(T)x$ is $\Gamma$-supercyclic for $T$ for every nonzero polynomial $p$.
\end{lemma}

\begin{proof} Since  $x$ is $\Gamma$-supercyclic for $T$, it follows from Lemma~\ref{lemma-spectre-vide} that $p(T)(\aorb(\Gamma x,T))$ is dense and we get the desired result by remarking that $\aorb(\Gamma p(T)x,T)\supset p(T)(\aorb(\Gamma x,T))$.
\end{proof}

We will now assume as in Theorem \ref{main-thm-sufficiency-coronna} that $\Gamma =[1,b]\T$ for some $b\geq 1$. The following result aims to divide $\Gamma$-supercyclic operators into two categories: the hypercyclic ones and the non-hypercyclic ones. Moreover, it gives some necessary properties satisfied by the non-hypercyclic ones.   

\begin{proposition}\label{prop1-thm1}Let $b\geq 1$. If $x$ is $[1,b]\T$-supercyclic then one of the two following conditions holds:
\begin{enumerate}\item $x$ is hypercyclic for $T$;
\item There exists $1<c\leq b$ such that $x$ is $[1,c]\T$-supercyclic but $\aeorb(\T x,T)\cap [1,c]\T x=\T x\cup c\T x$. In particular, $T$ is not hypercyclic.
\end{enumerate}
\end{proposition}

\begin{proof}We will need two claims.
\begin{fait}\label{claim2-lemma1}If there exist $\lambda\in ]1,b]$ and $n\ge 0$ such that $\lambda T^nx\in \overline{\text{Orb}}(\T x,T)$, then $x$ is $[1,\lambda]\T$-supercyclic for $T$.
\end{fait}

\noindent{}\emph{Proof of Claim \ref{claim2-lemma1}.} If $\lambda T^nx\in \overline{\text{Orb}}(\T x,T)$ then there exists $(n_k)\subset \N$ and $\gamma \in \T$ such that $\gamma T^{n_k}x\to \lambda T^nx$. It follows that for every $m\in \N$ and every $\mu >0$,
$$\frac{\mu \gamma }{\lambda}T^{n_k+m}x\to \mu T^{n+m}x$$
and thus
\begin{equation}\label{eq1-b-claim1-thm1}\aorb(\mu\T T^{n+m} x,T)\subset \aorb(\frac{\mu}{\lambda}\T T^{m} x,T).
\end{equation}

Let now $J\in \N$ be such that $(1/\lambda)^J\leq \lambda/b$. Then for every $\mu \in [1,b]$, there exists $0\leq j_{\mu}\leq J$ such that $ \frac{\mu}{\lambda^{j_{\mu}}}\in [1,\lambda]$. Thus it follows from (\ref{eq1-b-claim1-thm1}) that
$$\aorb(\mu \T T^{Jn} x,T)\subset \aorb(\mu \T T^{j_{\mu}n}  x,T) \subset \aorb(\frac{\mu}{\lambda^{j_{\mu}}} \T x,T).$$
Using Corollary \ref{cor0} we get
$$X=\bigcup_{\mu\in [1,b]}\aorb(\mu \T T^{Jn}  x,T)\subset \bigcup _{\nu\in [1,\lambda]}\aorb(\nu \T x,T).$$

\medskip{}
\begin{fait}\label{claim3-lemma1}Let $1\leq c < b$. If $x$ is $[1,\lambda]\T$-supercyclic for every $\lambda\in ]c,b]$ then $x$ is $[1,c]\T$-supercyclic.
\end{fait}

\noindent{}\emph{Proof of Claim~\ref{claim3-lemma1}.} Let $y\in X$ and $c<b$. Since $x$ is $[1,c+1/k]\T$-supercyclic for every $k$ large enough, there exists $\mu_k\in [1,c+1/k]\T$,  and $n_k$ such that
\[\|\mu_kT^{n_k}x-y\|\le \frac{1}{k}.\]
Up to take a subsequence, we may assume that $\mu_k\to \mu$ for some $\mu\in [1,c]\T$. We deduce that $\mu T^{n_k}x\to y$ and thus that $x$ is $[1,c]\T$-supercyclic.

\medskip{}
We now finish the proof of Proposition \ref{prop1-thm1}. Set
$$c=\inf\{\lambda\in [1,b]:\text{$x$ is $[1,\lambda]\T$-supercyclic}\}.$$
If $x$ is not hypercyclic, then $c>1$, because if not Claim~\ref{claim3-lemma1} implies that $x$ is $\T$-supercyclic hence hypercyclic, by Le\'on-M\"uller Theorem. Then, first, we deduce from Claim~\ref{claim2-lemma1} that for every $\lambda \in ]1,c[$, $\lambda x \notin \aorb(\T x,T)$. Moreover $cx$ must belong to $\aorb(\T x,T)$ because if not there exists $\varepsilon >0$ such that the interval $]x,(c+\varepsilon)x]$ is included in the complement of $\aorb(\T x,T)$ and then $(c+\varepsilon)x\notin \aorb([1,c]\T x,T)$, what contradicts the fact that $x$ is $[1,c]\T$-supercyclic. Thus, we have $\aorb(\T x,T)\cap [1,c]\T x=\T x\cup c\T x$.

Finally if we are in such a case, then $T$ fails to be hypercyclic. Indeed if $y\in X$ is hypercyclic for $T$ then, according to Corollary \ref{cor0}, there exists $\lambda \in [1,c]$ such that $y\in\aorb(\lambda\T x, T)$. From this and the hypercyclicity of $y$, we deduce that
$$X=\aorb(y,T)\subset \aorb(\lambda \T x,T).$$
By Le\'on-M\"uller Theorem again, $\lambda x$ would then be hypercyclic for $T$ and thus $x$ would be hypercyclic for $T$.

\end{proof}

The remaining of the proof will consist in showing that there cannot exist an operator $T\in L(X)$ admitting a $[1,b]\T$-supercyclic vector $x$ with $b>1$ such that
$$\aorb(\T x,T)\cap [1,b]\T =\T x \cup b\T x.$$
This will conclude the proof of the sufficiency part of Theorem~A in view of the previous result.
To do so we will show that if such an operator $T$ exists then we can build a certain continuous function $\Lambda$ from $\text{span}\{\orb(x,T)\}\setminus \{0\}$ into $\T$ inducing an homotopy in $\T$ between a single point and a closed path having nonzero winding number around $0$, what is known to be impossible. The construction of this continuous function $\Lambda$ will rely for any $y\in \text{span}\{\orb(x,T)\}\setminus \{0\}$ on the existence of a unique parameter $\lambda_y\in [1,b[$ such that $y\in\overline{\text{Orb}}(\lambda_y\T x,T)$. The existence of this parameter will be obtained for every $[1,b]$-supercyclic vector and thus for every element in $\text{span}\{\orb(x,T)\}\setminus \{0\}$ in view of Lemma \ref{lemma-P(T)}.

This will be done thanks to the following lemmas which help to understand how the orbit of $x$ approaches real multiples of a fixed $[1,b]$-supercyclic vector $y$.
We first remark that if $x$ is a $[1,b]\T$-supercyclic vector satisfying $\aorb(\T x,T)\cap [1,b]\T x=\T x \cup b\T x$ then every $[1,b]\T$-supercyclic vector satisfies this property.

\begin{lemma}\label{lemma4-thm1}If $x$ is $[1,b]\T$-supercyclic for $T$ but $\aeorb(\T x,T)\cap [1,b]\T x=\T x \cup b\T x$ then for every $[1,b]\T$-supercyclic vector $y$, we have $\aeorb(\T y,T)\cap [1,b]\T y=\T y \cup b\T y$.
\end{lemma}

\begin{proof}Let $y$ be a $[1,b]\T$-supercyclic vector for $T$. By Proposition \ref{prop1-thm1} there exists $1<c\leq b$ such that $y$ is $[1,c]\T$-supercyclic for $T$ and $\aorb(\T y,T)\cap [1,c]\T =\T y \cup c\T y$. It is enough to show that $c=b$. Let $\mu \in [1,b]$ be such that $y\in \aorb(\mu \T x,T)$. We have
$$X=\aorb([1,c]\T y,T)\subset \aorb([1,c]\T \mu x,T),$$
so that $\mu x$ and then $x$ are $[1,c]\T$-supercyclic for $T$, which is true if and only if $c=b$.
\end{proof}

We can also characterize the multiples of $x$ belonging to the orbit of $x$ itself.

\begin{lemma}\label{lemma5-thm1}Let $x$ be $[1,b]\T$-supercyclic for $T$ such that $\aeorb(\T x,T)\cap [1,b]\T x=\T x \cup b\T x$. If, for some $\mu >0$, $\mu x$ belongs to $\aeorb(\T x,T)$ then $\mu b^mx$ belongs to $\aeorb(\T x,T)$ for every $m\in \Z$.
\end{lemma}

\begin{proof}Since $bx \in \aorb(\T x,T)$ we get $\aorb(\T b^mx,T)\subset \aorb(\T b^{m-1}x,T)$ for every $m\ge 1$ so, if $\mu x \in \aorb(\T x,T)$ then we deduce that $\mu b^mx \in \aorb(\T b^mx,T)\subset \aorb(\T x,T)$ for every $m\geq 0$. Similarly, to prove that the latter holds also for $m<0$ it is enough to show that $\frac{1}{b}x \in \aorb(\T x,T)$. Now observe that $x$ is $[\frac{1}{b^2},\frac{1}{b}]\T$-supercyclic for $T$ so $\aorb(\T x,T)$ must contain an element of $[\frac{1}{b^2},\frac{1}{b}]\T x$. But, by the previous, if $\lambda x \in \aorb(\T x,T)\cap [\frac{1}{b^2},\frac{1}{b}]\T x$ for some $\lambda\in\R_+$ then $\lambda b^2 x \in \aorb(\T x,T)\cap [1,b]\T x$ hence $\lambda = 1/b^2$ or $\lambda =1/b$ by hypothesis. Finally if $\lambda = 1/b^2$ then $\frac{1}{b}x=\frac{1}{b^2}bx \in \aorb(\T x,T)$.
\end{proof}

Moreover, this characterization transfers to arbitrary $[1,b]\T$-supercyclic vectors for $T$.

\begin{lemma}\label{lemma6-thm1}If $x$ is $[1,b]\T$-supercyclic for $T$ but $\aeorb(\T x,T)\cap [1,b]\T x=\T x \cup b\T x$ then for every $[1,b]\T$-supercyclic vector $y$, we have $\aeorb(\T y,T)\cap \R _+y =\left\{b^n y:\,n\in\Z\right\}$.
\end{lemma}

\begin{proof}From Lemmas \ref{lemma4-thm1} and \ref{lemma5-thm1}, since $y\in \aorb(\T y,T)$,
$$\aorb(\T y,T)\cap \R_+y\supset \left\{b^ny:\,n\in \Z\right\}.$$
Let now $\mu y\in \aorb(\T y,T)\cap \R_+y$ with $\mu\in\R_+$. By Lemmas \ref{lemma4-thm1} and \ref{lemma5-thm1}, there exists $m\in \Z$ such that $\mu b^m y\in \aorb(\T y,T)\cap [1,b]y$. If $\mu \neq b^n$ for any $n\in \Z$, then we have a contradiction with Lemma \ref{lemma4-thm1}.
\end{proof}

Thanks to the previous lemmas, we are now able to describe completely the set $\aorb(\T x,T)\cap \R _+y$ in a unified way where the dependence on $y$ appears only through a single parameter $\lambda$.

\begin{proposition}\label{prop2-thm1}If $x$ is $[1,b]\T$-supercyclic for $T$ but $\aeorb(\T x,T)\cap [1,b]\T x=\T x \cup b\T x$ then for every $[1,b]\T$-supercyclic vector $y$, there exists $\lambda \in [1,b]$ such that
$$\aeorb(\T x,T)\cap \R _+y =\left\{\frac{b^n}{\lambda} y:\,n\in\Z\right\}.$$
\end{proposition}

\begin{proof}Let $y$ be $[1,b]\T$-supercyclic for $T$. Since $x$ is $[1,b]\T$-supercyclic for $T$ there exists $\lambda \in [1,b]$ such that $y\in \aorb(\lambda \T x,T)$. Then, Lemma \ref{lemma6-thm1} implies
$$\aorb(\T x,T)\supset \bigcup_{n\in \mathbb{Z}}\aorb(b^n\T x,T)\supset \left\{\frac{b^n}{\lambda}y:\,n\in\mathbb{Z}\right\}.$$
Let now $\mu y \in \aorb(\T x,T)$. Since $y$ is $[1,b]\T$-supercyclic there exists $\tau \in[1,b]$ such that $x\in \aorb(\tau \T y,T)$. Thus $\aorb(\tau \T y,T)\supset \{\mu y , y/\lambda\}$ hence by Lemma \ref{lemma6-thm1} $\mu=\tau b^n$ and $\frac{1}{\lambda}=\tau b^m$ for some $m,n$ and thus $\mu=b^{n-m}/\lambda$.
\end{proof}

\medskip{}

Let $x$ be $[1,b]\T$-supercyclic for $T$ such that $\aorb(\T x,T)\cap [1,b]\T x=\T x \cup b\T x$. Given a $[1,b]\T$-supercyclic vector  $y$, observe that if the $\lambda$ given by the previous proposition belongs to $]1,b[$ then it is unique. Otherwise, $\lambda =1$ and $\lambda =b$ works for $y$. Also note that if $\lambda \in ]1,b[$ then it is the unique $\lambda \in ]1,b[$ such that $y\in \aorb(\lambda \T x,T)$. Similarly, $\lambda \in \{1,b\}$ if and only if $y\in \aorb(\T x,T)$. 

\medskip{}

Let $\varphi:[1,b]\rightarrow \T$ be the parametrization of $\T$ given by $\varphi(t)=\exp(2i\pi\frac{t-1}{b-1})$. According to the previous observation and thanks to Lemma \ref{lemma-P(T)}, we can define an application $\Lambda:\text{span}\{\orb(x,T)\}\setminus\{0\}\rightarrow \T$ by

$$\Lambda (y)=\left\{\begin{array}{ll} \varphi(\lambda _y) & \mbox{ if $y\notin \aorb(\T x,T)$ where $\lambda _y$ is uniquely given by Proposition \ref{prop2-thm1}}\\
1 & \mbox{ if $y\in \aorb(\T x,T)$.}
\end{array}\right.$$

This application $\Lambda$ is well-defined according to the above observation. Moreover, we remark that for every $\lambda\in [1,b]$, we have $u\in \aorb(\lambda\T x, T)\cap \text{span}\{\orb(x,T)\}\setminus\{0\}$ if and only if $\Lambda(u)=\varphi(\lambda)$. It will play a crucial role to end up with a contradiction, assuming that such an operator $T$ exists.

\begin{corollary}\label{main-cor-thm1}Let $x$ be a $[1,b]\T$-supercyclic vector for $T$ such that $\aeorb(\T x,T)\cap [1,b]\T x=\T x \cup b\T x$. The following properties hold.
\begin{enumerate}\item $\Lambda$ is continuous;
\item $\Lambda (\mu T^nx)=\varphi (\mu)$ for every $n\geq 0$ and every $\mu \in [1,b]$.
\end{enumerate}
\end{corollary}

\begin{proof}(1) It is sufficient to prove that, for every sequence $(u_n)\subset \text{span}\{\orb(x,T)\}\setminus\{0\}$ and every $u\in \text{span}\{\orb(x,T)\}\setminus\{0\}$, if $u_n\to u$ then $\Lambda (u)$ is the only limit point of $(\Lambda (u_n))_n$. By compactness of $\T$, we can assume without loss of generality that $\Lambda (u_n)\rightarrow \alpha \in \T$ and we have to show that $\alpha =\Lambda (u)$. If $\Lambda (u_n)=1$ for infinitely many $n$ then first $\alpha =1$ and, second, infinitely many $u_n$ belongs to $\aorb(\T x,T)$. It follows that $u\in \aorb(\T x,T)$ so that $\Lambda (u)=1=\alpha$. If we are not in the previous case, then we can assume that $u_n \notin \aorb(\T x,T)$ for every $n$ so $\Lambda (u_n)=\varphi(\lambda _n)$, $n\geq 0$, with $\lambda _n \in ]1,b[$ and $u_n \in \aorb(\lambda _n \T x,T)$. By compactness we can assume that $\lambda _n \rightarrow \lambda \in [1,b]$ so that $u\in \aorb(\lambda \T x,T)$. By continuity of $\varphi$ it follows that $\varphi(\lambda)=\alpha=\Lambda (u)$.\\
(2) comes easily from the fact that $T^nx \in \aorb(\T x,T)$ for any $n\geq 0$ and the definition of $\Lambda$.
\end{proof}

We are now ready to finish the proof of Theorem \ref{main-thm-sufficiency-coronna}.

\begin{proof}[Proof of Theorem \ref{main-thm-sufficiency-coronna}] We assume by contradiction that $x$ is not hypercyclic for $T$. By Proposition \ref{prop1-thm1} we can assume that $x$ is a $[1,b]\T$-supercyclic vector for $T$ such that $\aorb(\T x,T)\cap [1,b]\T x=\T x \cup b\T x$, and thus that the application $\Lambda :\text{span}\{\orb(x,T)\}\setminus\{0\}\rightarrow \T$ introduced above is well-defined.

For every $y_0,y_1\in \text{span}\{\orb(x,T)\}$, we let $[y_0,y_1]:=\{(1-t)y_0+ty_1:t\in[0,1]\}$ and if $0\notin [y_0,y_1]$, we define the closed (continuous) curve $\gamma_{[y_0,y_1]}:[0,1]\rightarrow \T$ by $\gamma_{[y_0,y_1]}=\Lambda \circ \tilde{\gamma_{[y_0,y_1]}}$ where $\tilde{\gamma_{[y_0,y_1]}}:[0,1]\rightarrow \text{span}\{\orb(x,T)\}$ is given by
$$\tilde{\gamma_{[y_0,y_1]}}(t)=(1-t)y_0+t y_1.$$
Note that $0$ does not belong to the image of $\tilde{\gamma_{[T^n x,T^m x]}}$ for any $n,m\geq 0$, and that $\gamma_{[T^n x,T^m x]}$ is a closed continuous curve by Corollary \ref{main-cor-thm1}. Moreover we observe that $\gamma _{[T^n x,T^{n+1} x]}=\gamma _{[x,Tx]}$ for any $n\geq 0$. Indeed this comes from the definition of $\Lambda$ and from the fact that if $y\in \aorb(\lambda \T x,T)$ then $T^ny\in \aorb(\lambda \T x,T)$ for every $n\in\N$. So in particular, $\text{Ind}_0\gamma _{[T^n x,T^{n+1} x]}=\text{Ind}_0\gamma_{[x,T x]}$ for any $n\geq 0$, where $\text{Ind}_0\gamma$ stands for the winding number of a closed continuous curve $\gamma$ around $0$. On the other hand, for each $\theta \in [0,2\pi[$ and each $y\in \text{span}\{\orb(x,T)\}\setminus \{0\}$ we define the closed (continuous) curve $\gamma_{\theta,y}:[0,1]\rightarrow \T$ by $\gamma_{\theta,y}=\Lambda \circ \tilde{\gamma_{\theta,y}}$ where $\tilde{\gamma_{\theta,y}}:[0,1]\rightarrow \text{span}\{\orb(x,T)\}\setminus\{0\}$ is given by
$$\tilde{\gamma_{\theta,y}}(t)=e^{i\theta t}y.$$
It is again easily seen by definition of $\Lambda$ that $\gamma_{\theta,y}$ is the constant path equal to $\Lambda(y)$, therefore $\text{Ind}_0\gamma_{\theta,y}=0$. Similarly, we observe that $\text{Ind}_0\gamma_{[bx, x]}=-1$.

Now, using Lemma \ref{lemma6-thm1}, we deduce that $bx\in\aorb(\T x,T)\setminus \orb(\T x,T)$ because otherwise $\orb(\T x,T)$ would be contained in a finite dimensional space contradicting the $[1,b]\T$-supercyclicity of $x$. Then, by compactness of $\T$, there exists $\theta \in [0,2\pi[$ and $(n_k)_k \subset \N$ increasing such that $e^{i\theta}T^{n_k}x \rightarrow bx$ as $k$ tends to $\infty$. We assert that, up to take a subsequence, $(n_k)_k$ can be chosen in such a way that $\text{Ind}_0\gamma _{[e^{i\theta}T^{n_k}x,bx]}=0$. Indeed, if we assume by contradiction that for some $N\geq 0$ and every $k\geq N$ the winding number $\text{Ind}_0\gamma _{[e^{i\theta}T^{n_k}x,bx]}$ is nonzero, then for every $\lambda\in [1,b[$, 
\[[e^{i\theta}T^{n_k}x,bx]\cap \aorb(\lambda\T x,T)\ne \emptyset\]
for any $k\geq N$.
Yet for every $\varepsilon >0$, there exists $N_{\varepsilon}\geq N$ such that for every $k\geq N_{\varepsilon}$ $[e^{i\theta}T^{n_k}x,bx]\subset B(bx,\varepsilon)$. In other words, for any $\lambda \in [1,b[$, there exists a sequence $(y_n)$ converging to $bx$ such that for every $n\geq 0$, $y_n\in \aorb(\lambda \T x,T)$. Thus $bx\in \aorb(\lambda \T x,T)$ for every $\lambda \in [1,b[$, a contradiction with Proposition \ref{prop2-thm1}.

For the remaining of the proof, let $\theta \in [0,2\pi[$ and $(n_k)_k \subset \N$ increasing be such that for every $k\geq 0$, $\text{Ind}_0\gamma _{[e^{i\theta}T^{n_k}x,bx]}=0$.

Given any $n\geq 0$, we define $\tilde{\gamma_{n,\theta}}:[0,1]\rightarrow \text{span}\{x,\ldots,T^{n}x\}$ by
$$\tilde{\gamma_{n,\theta}}(s)=\left\{\begin{array}{l}
\tilde{\gamma_{[T^jx,T^{j+1}x]}}((n+3)s-j))\quad \mbox{if $\frac{j}{n+3}\leq s\leq \frac{j+1}{n+3}$,\quad$0\leq j\leq n-1$}\\[4pt]
\tilde{\gamma_{\theta,T^{n}x}}((n+3)s-n)\quad \mbox{if $\frac{n}{n+3}\leq s\leq \frac{n+1}{n+3}$}\\[4pt]
\tilde{\gamma_{[e^{i\theta}T^{n}x,bx]}}((n+3)s-(n+1))\quad \mbox{if $\frac{n+1}{n+3}\leq s\leq \frac{n+2}{n+3}$}\\
\tilde{\gamma_{[bx,x]}}((n+3)s-(n+2))\quad \mbox{if $\frac{n+2}{n+3}\leq s\leq 1$}.\\
\end{array}\right.$$
By construction, one easily notices that $\tilde{\gamma_{n,\theta}}$ does never take the value zero.
Then we can set $\gamma_{n,\theta}=\Lambda \circ \tilde{\gamma_{n,\theta}}$ for every $n\geq 0$. Moreover, since $\text{span}\{\orb(x,T)\}$ is infinite dimensional we can retract, staying in $\text{span}\{\orb(x,T)\}\setminus \{0\}$, the closed curve $\tilde{\gamma_{n,\theta}}$ onto some $T^{m}x\in \text{span}\{\orb(x,T)\}\setminus \text{span}\{x,\ldots,T^{n}x\}$ for every $n\geq 0$ and some $m>n$, and thus build an homotopy of closed curves in $\T$ such that $\gamma_{n,\theta}$ is homotopic to the constant path $\Lambda(T^{m}x)$. Thus $\text{Ind}_0\gamma _{n,\theta}=0$ for every $n\geq 0$.

With $\theta$ and $(n_k)_k$ as above and as a consequence of the observations made at the beginning at the proof, we deduce that for every $k\geq 0$
\begin{eqnarray*}0 & = & \text{Ind}_0\gamma_{n_k,\theta}\\
& = & \sum _{j=0}^{n_k-1}\text{Ind}_0\gamma _{[T^jx,T^{j+1}x]}+\text{Ind}_0\gamma_{\theta,T^{n_k}x}+\text{Ind}_0\gamma_{[e^{i\theta}T^{n_k} x,bx]}+\text{Ind}_0\gamma _{[bx,x]}\\
& = & n_k\text{Ind}_0\gamma_{[x,Tx]}+0+0-1,\\
\end{eqnarray*}
and it follows that $n_k\text{Ind}_0\gamma_{[x,Tx]}=1$ for any $k\geq 0$, which is impossible since $n_k$ tends to $\infty$.
\end{proof}

\section{Proof of Theorem~B}

The aim of this section is to prove Theorem~B. We begin by proving the sufficiency.

\begin{theorem}\label{Pprime-suff}
Let $T\in L(X)$ and let $\Gamma\subset \mathbb{C}$ be such that $\overline{\Gamma\mathbb{T}}\setminus\{0\}$  is bounded and bounded away from zero with an empty interior. If the set $\text{\emph{Orb}}(\Gamma\mathbb{T} x,T)$ is somewhere dense in $X$, then $x$ is a hypercyclic vector for $T$.
\end{theorem}
\begin{proof}
Without loss of generality, we can suppose that $0\notin \Gamma$.
Let
$$\Lambda_n:=\{\lambda\in \mathbb{R}^+:T^nx\in\overline{\text{Orb}}(\lambda \mathbb{T} x,T)\}.$$
By definition, the sequence $(\Lambda_n)$ is a non-decreasing sequence and $\Lambda_n\cup\{0\}$ is a closed set.
Moreover, if $\lambda\in \Lambda_n$ then for every $\varepsilon>0$, there exists $m\ge n$ and $\theta\in [0,2\pi]$ such that $\|\lambda e^{i\theta} T^mx-T^nx\|<\varepsilon$. We can assume that $m\ge n$ because otherwise we would have $T^nx\in \text{span}\{x,\dots,T^{n-1}x\}$ and thus $\text{Orb}(\Gamma\mathbb{T} x,T)$ would not be somewhere dense. This implies that for every $n\ge 0$, if $\lambda,\lambda'\in \Lambda_n$ then the product $\lambda\lambda'\in \Lambda_n$. Indeed, if $\lambda,\lambda'\in \Lambda_n$ then for every $\varepsilon>0$, there exists $m\ge n$ and $\theta\in [0,2\pi]$ such that $\|\lambda e^{i\theta} T^mx-T^nx\|<\frac{\varepsilon}{2}$ and there exists $m'\ge 0$ and $\theta'\in [0,2\pi]$ such that $\|\lambda' e^{i\theta'}T^{m'}x-T^nx\|<\frac{\varepsilon}{2} \lambda^{-1}\|T^{m-n}\|^{-1}$. We then get
\begin{align*}
&\|\lambda\lambda'e^{i(\theta+\theta')}T^{m+m'-n}x-T^{n}x\|\\
&\quad\le \lambda \|\lambda'e^{i\theta'} T^{m+m'-n}x-T^{m}x\|+\|\lambda e^{i\theta}T^mx-T^nx\|\\
&\quad\le  \lambda \|T^{m-n}\| \|\lambda' e^{i\theta'}T^{m'}x-T^{n}x\|+\|\lambda e^{i\theta} T^mx-T^nx\|\le\varepsilon.
\end{align*}
In particular, if $\lambda\in \Lambda_n$ then $\lambda^k\in \Lambda_n$ for every $k\ge 1$.\\

The idea of the proof of this theorem consists in showing that if $\text{Orb}(\Gamma\mathbb{T} x,T)$ is somewhere dense in $X$ and $\overline{\Gamma\mathbb{T}}$ has an empty interior then
\[\overline{\bigcup_n\Lambda_{n}}\supset [1,+\infty[ \quad\text{or}\quad \overline{\bigcup_n\Lambda_{n}}\supset [0,1]\]
and that if one of these inclusions holds then $\text{Orb}(\mathbb{T} x,T)$ is also somewhere dense and thus $x$ is hypercyclic by the generalized Bourdon-Feldman Theorem given in \cite[Theorem 3.13]{Bay} and stated in the introduction.

To this end, we consider a non-empty open set $U$ such that $U\subset \overline{\text{Orb}}(\Gamma\mathbb{T} x,T)$. We deduce that for every $y\in U$, there exists $\gamma\in |\overline{\Gamma}|$ such that $y\in \aorb(\gamma\mathbb{T} x,T)$ where $|\overline{\Gamma}|=\{|\gamma|:\gamma\in \overline{\Gamma}\}$ is bounded  and bounded away from zero. Given $y\in U$, we let $\gamma(y):=\inf\{\gamma\in |\overline{\Gamma}|: y\in \aorb(\gamma\mathbb{T} x,T)\}$. In particular, we have $y\in \aorb(\gamma(y)\mathbb{T} x,T)$ and we remark that if $y_n\to y$ then $\liminf \gamma(y_n)\ge \gamma(y)$.\\
Let $M=\sup_{y\in U} \gamma(y)$ and $\varepsilon>0$. There exists $y\in U$ such that $\gamma(y)> M-\varepsilon$. Since there exists a sequence $(n_k)$ such that $\gamma(y)T^{n_k}x\to y$ and $\gamma(y)T^{n_k}x\in U$, we deduce that there exists $n\ge 0$ such that 
\[\gamma(y)T^nx\in U\quad\text{and}\quad\gamma(\gamma(y)T^nx)> M-2\varepsilon.\]
 We now prove that $\Lambda_n$ contains a limit point belonging to $[\frac{M-2\varepsilon}{M},\frac{M}{M-\varepsilon}]$. Since $U$ is a non-empty open set, there exists $\eta>0$ such that the set $\{\lambda'\gamma(y)T^{n}x:1\le \lambda'<1+\eta\}$ is included in $U$. We construct by induction a sequence $(\lambda_k)\subset ]1,1+\eta[$ tending to $1$ and a sequence $(\gamma_k)_k\subset |\overline{\Gamma}|$ such that for every $k\ne j$,
\[\frac{\gamma_k}{\lambda_k \gamma(y)}\in \Lambda_n \quad\text{and}\quad \frac{\gamma_k}{\lambda_k \gamma(y)}\ne \frac{\gamma_j}{\lambda_j \gamma(y)}.\]
Let $\lambda_1\in ]1,1+\eta[$. Since $\lambda_1 \gamma(y)T^{n}x\in U$, there exists $\gamma_1\in |\overline{\Gamma}|$ such that \[\lambda_1 \gamma(y)T^{n}x\in \aorb(\gamma_1\mathbb{T} x,T)\] and we deduce that $\frac{\gamma_1}{\lambda_1 \gamma(y)}\in \Lambda_n$. Assume that $\lambda_1,\cdots,\lambda_k$ have been fixed. We then choose $\lambda_{k+1}\in ]1,1+\frac{\eta}{k+1}[$ such that for every $j\le k$, $\frac{\lambda_{k+1}\gamma_j}{\lambda_j}\notin |\overline{\Gamma}|$. Such a constant $\lambda_{k+1}$ exists because $|\overline{\Gamma}|$ has an empty interior. Therefore, since $\lambda_{k+1}\gamma(y)T^{n}x\in U$, there exists $\gamma_{k+1}\in |\overline{\Gamma}|$ such that \[\lambda_{k+1} \gamma(y)T^{n}x\in \aorb(\gamma_{k+1}\mathbb{T} x,T)\]
and we deduce that $\frac{\gamma_{k+1}}{\lambda_{k+1} \gamma(y)}\in \Lambda_n$ and that for every $j\le k$,
\[\frac{\gamma_{k+1}}{\lambda_{k+1} \gamma(y)}\ne \frac{\gamma_j}{\lambda_j \gamma(y)}\]
since $\frac{\lambda_{k+1}\gamma_j}{\lambda_j}\notin |\overline{\Gamma}|$.
Finally, since $|\overline{\Gamma}|$ is compact and $\lambda_n\to 1$, there exists an increasing sequence $(n_k)$ and $\gamma\in |\overline{\Gamma}|$ such that $\frac{\gamma_{n_k}}{\lambda_{n_k}\gamma(y)}\to \frac{\gamma}{\gamma(y)}$ and thus $\frac{\gamma}{\gamma(y)}\in \Lambda_n$. Moreover, we have $\gamma\ge \gamma(\gamma(y)T^nx)$ and thus $\frac{\gamma}{\gamma(y)}\in [\frac{M-2\varepsilon}{M},\frac{M}{M-\varepsilon}]$. We conclude that 
$\frac{\gamma}{\gamma(y)}$ is a limit point of $\Lambda_n$ belonging to $[\frac{M-2\varepsilon}{M},\frac{M}{M-\varepsilon}]$.\\

In other words, we have proved that for every $\varepsilon>0$, there exists $n\ge 0$ such that $\Lambda_n$ contains a limit point in $]1-\varepsilon,1+\varepsilon[$. In particular, this implies that $1$ is a limit point of $\bigcup_{n}\Lambda_n$. Since each power of an element of $\Lambda_n$ still belongs to $\Lambda_n$, we deduce that
\[\overline{\bigcup_n\Lambda_{n}}\supset [1,+\infty[ \quad\text{or}\quad \overline{\bigcup_n\Lambda_{n}}\supset [0,1].\]
We now show that each of these inclusions implies that $\overline{\text{Orb}}(\mathbb{T} x,T)$ is somewhere dense. We first remark that if we let $U_{\infty}=U\backslash \text{Orb}(\overline{\Gamma}\mathbb{T} x,T)$ then since the interior of $\text{Orb}( \overline{\Gamma}\mathbb{T} x,T)$ is empty, we have $\overline{U}_{\infty}\supset U$. It thus suffices to prove that $\overline{\text{Orb}}(\mathbb{T} x,T)$ contains a nonzero multiple of $U_{\infty}$ in order to conclude.\\
Assume that $\overline{\bigcup_n\Lambda_{n}}\supset [0,1]$ and let $c=\inf |\Gamma|$.
By definition, for every $y\in U_{\infty}$, there exist an increasing sequence $(n_k)$, $\gamma\in |\overline{\Gamma}|$ and $\theta\in [0,2\mathbb{T}]$ such that \[\gamma e^{i\theta}T^{n_k}x\to y.\] Since $\frac{\gamma}{c}\ge 1$ and $\overline{\bigcup_n\Lambda_{n}}\supset [0,1]$, there also exists a sequence $\lambda_{k}\in \Lambda_{n_k}$ such that $\frac{1}{\lambda_{k}}\to \frac{\gamma}{c}$. We then deduce that $\frac{e^{i\theta}}{\lambda_k}T^{n_k}x\to \frac{y}{c}$ and since $\frac{e^{i\theta}}{\lambda_k}T^{n_k}x\in \overline{\text{Orb}}(\mathbb{T} x,T)$, we conclude that $\frac{y}{c}\in \overline{\text{Orb}}(\mathbb{T} x,T)$. If $\overline{\bigcup_n\Lambda_{n}}\supset [1,\infty[$, we get, by applying the same method, that $\overline{\text{Orb}}(\mathbb{T} x,T)\supset \frac{1}{d}U_{\infty}$ where $d=\sup |\Gamma|$. The desired result follows.
\end{proof}

We now show the necessity part. It is enough to build an operator $T$ acting on some Banach space $X$ such that there exists $x\in X$ with $\aorb(\Gamma x,T)$ somewhere dense in $X$ but $\aorb(x,T)$ non dense in $X$.

\begin{proposition}\label{Pprime-nec}Let $\Gamma \subset \C$ be non-empty. We assume that for every complex Banach space $X$, every $T\in L(X)$ and every $x\in X$, if $\aeorb(\Gamma x,T)$ is somewhere dense in $X$, then $x$ is hypercyclic for $T$. Then $\Gamma\subset \mathbb{C}$ is such that $\overline{\Gamma\mathbb{T}}\setminus\{0\}$  is bounded and bounded away from zero with an empty interior.
\end{proposition}

\begin{proof}
If $\Gamma\setminus\{0\}$ is not bounded or not bounded away from 0 then counterexamples are given in Section 2.
Let then $\Gamma\setminus\{0\}$ be a bounded, bounded away from $0$ subset of $\C$ such that the interior of $\overline{\Gamma \T}$ is non-empty.
By \cite[Theorem 2.1 and Theorem 2.2.(b)]{Berbonper}, there exists an $\R^{+}$-supercyclic operator $T=e^{i\theta}\oplus \tilde{T}$ acting on a Banach space $X=\C\oplus Y$ with $\R^+$-supercyclic vector $(1,y)$. Clearly $(1,y)$ is not hypercyclic for $T$.
Let $V$ be a non-empty open subset in $Y$ and $U\subset \overline{\Gamma\T}\times V$ nonempty and open. We intend to prove that $U$ is in the interior of $\aorb(\Gamma (1,y),T)$.\\
Let $(a,x)\in U$. Since $a\in\overline{\Gamma\T}$, there exists $(\gamma_k)_{k\in\N}\subset \Gamma$ and $\mu\in[0,2\pi[$ such that $\gamma_k\to \vert a\vert e^{i\mu}$.
Moreover, since $(1,y)$ is $\R^+$-supercyclic for $T$, there also exist an increasing sequence $(n_k)_{k\in\N}$ and $(\lambda_{k})_{k\in\N}\subset\R^+$ such that
$$\lambda_k e^{i\theta n_k}\to a e^{-i\mu}\text{ et }\lambda_k \tilde{T}^{n_k}y\to x e^{-i\mu}.$$
From this we deduce that $\lambda_k\to \vert a\vert$ and thus
$$e^{i\theta n_k}\to \frac{a}{\vert a\vert e^{i\mu}} \text{ et }\tilde{T}^{n_k}y\to \frac{x}{\vert a\vert e^{i\mu}}.$$
Finally, reintroducing $\gamma_k$, we obtain 
$$\gamma_k e^{i\theta n_k}\to a \text{ et }\gamma_k \tilde{T}^{n_k}y\to x.$$
Hence $\gamma_k T^{n_k}(1,y)\to(a,x)$ which had to be shown.
\end{proof}

Theorem~B follows from the combination of Theorem \ref{Pprime-suff} and Proposition \ref{Pprime-nec}.

\section{Proof of Theorem~C}

In a matter of convenience we recall Theorem~C.

\begin{theoC}Let $X$ be a complex Banach space.
\begin{enumerate}
\item For every $\theta\in\R$, $\Gamma \subset \C$ satisfies the property: "For every $T\in L(X)$ with $\sigma_p(T^*)=\{e^{i\theta}\}$ and every $x\in X$
$$x\text{ is }\Gamma\text{-supercyclic for }T\text{ if and only if }x\text{ is supercyclic for }T\text{"}$$
if and only if $\Gamma G_{\theta}$ is dense in $\C$, where $G_{\theta}$ stands for the subgroup of $\T$ generated by $e^{i\theta}$.
\item For any $r\neq 1$ and any $\theta\in\R$, there exist $T\in L(X)$ with $\sigma_p(T^*)=\{re^{i\theta}\}$ and $\Gamma \subset \C$ satisfying $\overline{\Gamma G_{\theta}}=\C$, such that $T$ is supercyclic but not $\Gamma$-supercyclic.
\end{enumerate}
\end{theoC}

For the proof of (1) we will use the following lemma which is reminiscent from Shkarin's proof of \cite[Proposition 5.1]{Shkauniv} when $\theta$ is such that $\overline{G_{\theta}}=\T$. The proof for $\theta\in\R$ works along the same lines.
\begin{lemma}\label{shkarin-remi}
If $T$ is a supercyclic operator with $\sigma_p(T^*)=\{e^{i\theta}\}$ where $\theta\in\R$ then there exists $f\in X^*$ such that $\eorb(x,T)$ is dense in $\{y\in X: f(y)\in G_{\theta}\}$.
\end{lemma}

\begin{proof}[Proof of Theorem~C](1) Let $\theta\in\R$ and assume that $\Gamma G_{\theta}$ is dense in $\C$. If $x$ is a supercyclic vector for an operator $T$ with $\sigma_p(T^*)=\{e^{i\theta}\}$ then we know thanks to Lemma \ref{shkarin-remi} that there exists $f\in X^*$ such that $\orb(x,T)$ is dense in $\{y\in X: f(y)\in G_{\theta}\}$. Since $\overline{\Gamma G_{\theta}}=\C$, $\Gamma\{y\in X:  f(y)\in G_{\theta}\}$ is dense in $X$ and thus $\Gamma\orb(x,T)$ is also dense in $X$. This proves the sufficiency.

For the necessity part,  let $\theta\in\R$ and assume that $\Gamma G_{\theta}$ is not dense in $\C$. Let us consider the operator $R:=e^{-i\theta}Id$ on $\C$. Then, $\sigma_p(R^*)=\{e^{i\theta}\}$ and it is clear that $1$ is supercyclic for $R$ while $\Gamma\orb(1,R)\subset\Gamma G_{-\theta}$ is not dense in $\C$. Indeed observe that $G_{-\theta}=G_{\theta}$ when $\theta\in \pi \Q$, and that $\overline{ G_{-\theta}}=\overline{G_{\theta}}=\T$ when $\theta\in \pi(\R \setminus\Q)$. Now we decompose $X$ as a sum $X= \C \oplus Y$ and consider $T:=R \oplus \tilde{T}$ where $\tilde{T}: Y\to Y$ satisfies Kitai Criterion (or the Hypercyclicity Criterion along the whole sequence of integers). Then, it is clear that $\sigma_p(T^*)=\{e^{i\theta}\}$ and that $T$ is not $\Gamma$-supercyclic. Moreover, if $G_{\theta}$ is not dense in $\T$ then it is not difficult to check that $T$ is supercyclic thanks to Ansari Theorem \cite{Ansarihyp} for example. Finally, if $G_{\theta}$ is dense in $\T$ then the proof of the supercyclicity of $T$ is similar to the proof of \cite[Theorem 2.2.(b)]{Berbonper}.

(2) Let $r>0$ with  $r\neq1$ and $\theta\in\R$. Up to write $X=\C \oplus Y$ and consider $T=R\oplus \tilde{T}$ with $\tilde{T}$ satisfying Kitai Criterion, one can assume that $X=\C$ and we only have to exhibit $\Gamma \subset \C$ satisfying $\overline{\Gamma G_{\theta}}=\C$ and $R:\C \to \C$ supercyclic but not $\Gamma$-supercyclic, with $\sigma_p(R^*)=\{re^{i\theta}\}$ (see \cite[Theorem 2.2.(b)]{Berbonper} as previously). Then consider the operator $R:=re^{-i\theta}Id$ acting on $\C$ and set $\Gamma=\{r^te^{-it\theta};\,t\in \R\}$.
We first remark that $1$ is supercyclic for $R$ and $\sigma_p(R^*)=\{re^{i\theta}\}$. By contradiction we assume that $R$ is $\Gamma$-supercyclic. In this case, there exist a non-decreasing sequence $(n_k)_{k\in\N}$ of integers and $(\gamma_k)_{k\in\N}\subset\Gamma$ such that
$$\gamma_k r^{n_k}e^{-i\theta n_k}\to -1.$$
Writing $\gamma_k=r^{t_k}e^{-it_k\theta}$ for some $(t_k)_{k\in \N}\subset \R$, we deduce that 
$$e^{-i\theta (t_k+n_k)}\to -1\text{ and }r^{t_k+n_k}\to 1.$$
It follows that $t_k+n_k\to 0$ and $-1=\lim_{k\to\infty}e^{-i\theta (t_k+n_k)}=1$, a contradiction.

\end{proof}

\bibliographystyle{plain}
\bibliography{biblio}

\begin{thebibliography}{10}

\bibitem{Ansarihyp}
S.I. Ansari.
\newblock Hypercyclic and cyclic vectors.
\newblock {\em J. Funct. Anal.}, 128(2):374 -- 383, 1995.

\bibitem{Baymathyppb}
F.~Bayart and {\'E}.~Matheron.
\newblock Hypercyclic operators failing the hypercyclicity criterion on
  classical banach spaces.
\newblock {\em J. Funct. Anal.}, 250(2):426--441, 2007.

\bibitem{Bay}
F.~Bayart and {\'E}.~Matheron.
\newblock {\em Dynamics of linear operators}.
\newblock Cambridge tracts in mathematics. Cambridge University Press, 2009.

\bibitem{Berbonper}
T.~Berm{\'u}dez, A.~Bonilla, and A.~Peris.
\newblock {$\mathbb C$}-supercyclic versus {$\mathbb R^+$}-supercyclic
  operators.
\newblock {\em Arch. Math. (Basel)}, 79(2):125--130, 2002.

\bibitem{Berbonpe}
T.~Berm{\'u}dez, A.~Bonilla, and A.~Peris.
\newblock On hypercyclicity and supercyclicity criteria.
\newblock {\em Bull. Aust. Math. Soc.}, 70:45--54, 7 2004.

\bibitem{Bes}
J.P. B\`es.
\newblock {\em Three problems on hypercyclicity operators}.
\newblock Kent State University, (USA), 1998.
\newblock Ph.D. Thesis.

\bibitem{Besper}
J.P. B\`es and A.~Peris.
\newblock Hereditarily hypercyclic operators.
\newblock {\em J. Funct. Anal.}, 167(1):94--112, 1999.

\bibitem{Birk}
G.D. Birkhoff.
\newblock D\'emonstration d'un th\'eor\`eme \'el\'ementaire sur les fonctions
  enti\`eres.
\newblock {\em C. R. Acad. Sci. Paris}, 189:473--475, 1929.

\bibitem{Boufel}
P.~S. Bourdon and N.~S. Feldman.
\newblock Somewhere dense orbits are everywhere dense.
\newblock {\em Indiana Univ. Math. J.}, 52(3):811--819, 2003.

\bibitem{Costakisconjherrero}
G.~Costakis.
\newblock On a conjecture of {D}.\ {H}errero concerning hypercyclic operators.
\newblock {\em C. R. Acad. Sci. Paris S\'er. I Math.}, 330(3):179--182, 2000.

\bibitem{Feldmanperturbations}
N.~S. Feldman.
\newblock Perturbations of hypercyclic vectors.
\newblock {\em J. Math. Anal. Appl.}, 273(1):67--74, 2002.

\bibitem{Grope}
K.-G. Grosse-Erdmann and A.~Peris.
\newblock {\em Linear Chaos}.
\newblock Universitext Series. Springer, 2011.

\bibitem{Herrero}
D.A. Herrero.
\newblock Hypercyclic operators and chaos.
\newblock {\em J. Operator Theory}, 28(1):93--103, 1992.

\bibitem{Hilwal}
H.M. Hilden and L.J. Wallen.
\newblock Some cyclic and non-cyclic vectors of certain operators.
\newblock {\em Indiana Univ. Math. J.}, 23:557--565, 1973/74.

\bibitem{Kitai}
C.~Kitai.
\newblock {\em I{nvariant} {closed} {sets} {for} {linear} {operators}}.
\newblock ProQuest LLC, Ann Arbor, MI, 1982.
\newblock Thesis (Ph.D.)--University of Toronto (Canada).

\bibitem{Leon}
F.~Le\'on-Saavedra.
\newblock The positive supercyclicity theorem.
\newblock {\em Extracta Mathematicae}, 19(1):145--149, 2004.

\bibitem{LeonMuller}
F.~Le\'on-Saavedra and V.~M\"uller.
\newblock Rotations of hypercyclic and supercyclic operators.
\newblock {\em Integral Equations Operator Theory}, 50(3):385--391, 2004.

\bibitem{MacLane}
G.R. MacLane.
\newblock Sequences of derivatives and normal families.
\newblock {\em J. Analyse Math.}, 2:72--87, 1952.

\bibitem{Matheronsub}
{\'E}.~Matheron.
\newblock Subsemigroups of transitive semigroups.
\newblock {\em Ergodic Theory Dynam. Systems}, 32(3):1043--1071, 2012.

\bibitem{Montsalas}
A.~Montes-Rodr{\'i}guez and H.N. Salas.
\newblock Supercyclic subspaces: Spectral theory and weighted shifts.
\newblock {\em Adv. Math.}, 163(1):74 -- 134, 2001.

\bibitem{Permultihyp}
A.~Peris.
\newblock Multi-hypercyclic operators are hypercyclic.
\newblock {\em Mathematische Zeitschrift}, 236(4):779--786, 2001.

\bibitem{Rolewicz}
S.~Rolewicz.
\newblock On orbits of elements.
\newblock {\em Studia Math.}, 32:17--22, 1969.

\bibitem{Salas}
H.N. Salas.
\newblock Supercyclicity and weighted shifts.
\newblock {\em Studia Math.}, 135(1):55--74, 1999.

\bibitem{Shkauniv}
S.~Shkarin.
\newblock Universal elements for non-linear operators and their applications.
\newblock {\em J. Math. Anal. Appl.}, 348(1):193--210, 2008.

\end{thebibliography}

\end{document}